\documentclass{article}
\usepackage{amssymb}
\usepackage{latexsym}
\usepackage{amsmath}
\usepackage{mathrsfs}
\usepackage{color}
\usepackage{CJK}
\usepackage{graphicx}

\newtheorem{theorem}{Theorem}[section]

\newtheorem{lemma}{Lemma}[section]

\newtheorem{remark}{Remark}[section]

\numberwithin{equation}{section}
\newenvironment{proof}{\medskip\par\noindent{\bf Proof.}\ }{\qquad
\raisebox{-0.5mm}{\rule{1.5mm}{4mm}}\vspace{6pt}}

\newcommand{\bbr}{\mathbb{R}}
\newcommand{\bbrn}{\mathbb{R}^N}
\newcommand{\h}{H^2(\bbrn)}

\newcommand{\bbn}{\mathbb{N}}

\newcommand{\ve}{\varepsilon}

\begin{document}
\title
{\Large\bf On a biharmonic equations with steep potential well and indefinite potential}%

\author{
Yisheng Huang$^{a},$\thanks{E-mail address: yishengh@suda.edu.cn(Yisheng Huang)}\quad
Zeng Liu$^{b},$\thanks{E-mail address: luckliuz@163.com(Zeng Liu)}\quad
Yuanze Wu$^{c}$\thanks{Corresponding
author. E-mail address: wuyz850306@cumt.edu.cn (Yuanze Wu).}\\%
\footnotesize$^{a}${\em  Department of Mathematics, Soochow University, Suzhou 215006, P.R. China }\\%
\footnotesize$^{b}${\em  Department of Mathematics, Suzhou University of Science and Technology, Suzhou 215009, P.R. China}\\
\footnotesize$^{c}${\em  College of Sciences, China University of Mining and Technology, Xuzhou 221116, P.R. China }}%
\date{}
\maketitle

\noindent{\bf Abstract:} In this paper, we study the following biharmonic equations:%
$$
\left\{\aligned&\Delta^2u-a_0\Delta u+(\lambda b(x)+b_0)u=f(u)&\text{ in }\bbr^N,\\%
&u\in\h,\endaligned\right.\eqno{(\mathcal{P}_{\lambda})}%
$$
where $N\geq3$, $a_0,b_0\in\bbr$ are two constants, $\lambda>0$ is a parameter, $b(x)\geq0$ is a potential well and $f(t)\in C(\bbr)$ is subcritical and superlinear or asymptotically linear at infinity.  By the Gagliardo-Nirenberg inequality, we make some observations on the operator $\Delta^2-a_0\Delta+\lambda b(x)+b_0$ in $\h$.  Based on these observations, we give a new variational setting to $(\mathcal{P}_{\lambda})$ for $a_0<0$.  With this new variational setting in hands, we establish some new existence results of the nontrivial solutions to $(\mathcal{P}_{\lambda})$ for all $a_0, b_0\in\bbr$ with $\lambda$ sufficiently large by the variational method.  The concentration behavior of the nontrivial solutions as $\lambda\to+\infty$ is also obtained.  It is worth to point out that it seems to be the first time that the nontrivial solution of $(\mathcal{P}_{\lambda})$ is obtained in the case of $a_0<0$.%

\vspace{6mm} \noindent{\bf Keywords:} Variational method; Biharmonic equation; Potential well; Indefinite problem.%

\vspace{6mm}\noindent {\bf AMS} Subject Classification 2010: 35B38; 35B40; 35J10; 35J20.%

\section{Introduction}
In this paper, we study the following biharmonic equations:%
\begin{equation}\label{eq1000}
\left\{\aligned&\Delta^2u-a_0\Delta u+V(x)u=f(x,u)&\text{ in }\bbr^N,\\%
&u\in\h,\endaligned\right.%
\end{equation}
where $N\geq3$, $a_0\in\bbr$ is a constant and $\lambda>0$ is a parameter.  $V(x)$ and $f(x,u)$ satisfy some conditions to be specified later.%

The biharmonic equations in a bounded domain are generally regarded as a mathematical modeling, which can describe some phenomena appeared in physics, engineering and other sciences.  For example, the problem of nonlinear oscillation in a suspension bridge \cite{LM90,MW90} and the problem of the static deflection of an elastic plate in a fluid \cite{AD02}.  Due to such applications, the existence and multiplicity of nontrivial solutions for the biharmonic
equations in a bounded domain have been extensively studied in the past two decades, We refer the readers to \cite{H14,HW14,LH10,RTZ09,ZL05} and the references therein.  Most of the literatures were devoted to the following Dirichlet--Navier type boundary value problem:%
$$
\left\{\aligned&\Delta^2u-\alpha\Delta u=g(x,u)&\quad\text{in }\Omega,\\%
&u=\Delta u=0&\quad\text{on }\partial\Omega,\endaligned\right.\eqno{(\mathcal{P}_{\alpha})}%
$$
where $\Omega\subset\bbrn$ is bounded domain with smooth boundary, $\alpha>-\mu_1$ is a parameter and $\mu_1$ is the first eigenvalue of $-\Delta$ in $L^2(\Omega)$.  In particular, the existence of sign-changing solutions to $(\mathcal{P}_{\alpha})$ was obtained in \cite{LH10,RTZ09,ZW08} when $g(x,u)$ is subcritical and superlinear or asymptotically linear at infinity.%

In recent years, the study on Problem~\eqref{eq1000}, i.e. the biharmonic equations in the whole space $\bbrn$, has begun to attract much attention.
We refer the readers to \cite{CDM14,DDWW14,DS15,GG06,GHZ15,WS09,WZ12,YT13} and the references therein.
In these literatures, various existence results of the nontrivial solutions to Problem~\eqref{eq1000} were established by the variational method in the case of $a_0\geq0$.
Indeed, in the case of $a_0\geq0$, under some suitable conditions on $V(x)$ and $f(x,u)$, one can give a variational setting to Problem~\eqref{eq1000},
as the harmonic equations in the whole space, $\bbrn$ in the following Hilbert space%
\begin{equation*}
\mathcal{V}=\{u\in\h\mid \int_{\bbr^N}V^+(x)u^2dx<+\infty\},
\end{equation*}
where $V^+(x)=\max\{V(x),0\}$, the inner product and the corresponding norm are respectively given by%
\begin{equation*}
\langle u,v\rangle_{\mathcal{V}}=\int_{\bbr^N}(\Delta u\Delta v+a_0\nabla u\nabla v+V^+(x)uv)dx\quad\text{and}\quad\|u\|=\langle u,u\rangle_{\mathcal{V}}^{\frac12}.%
\end{equation*}
Thus, the variational method can be used to find the nontrivial solutions of Problem~\eqref{eq1000}, see for example \cite{CDM14,DS15,WS09,WZ12,YT13} and the references therein.

If $a_0<0$ then $\mathcal{V}$ with $\langle u,v\rangle_{\mathcal{V}}$ may not be a Hilbert space,
since the bilinear operator $\langle u,v\rangle_{\mathcal{V}}$ may not be an inner product in $\mathcal{V}$ for $V(x)\neq 0$ in general.  This is quite different from the situation of $V(x)=0$. Indeed, for example, if we consider the problem $(\mathcal{P}_\alpha)$ in a bounded $\Omega\subset\bbrn$, then $\alpha$ can take negative value since  the operator $\Delta^2-\alpha\Delta$ is compact in $L^2(\Omega)$ and the spectrum of $\Delta^2-\alpha\Delta$ in $L^2(\Omega)$ are the eigenvalues  $\{\mu_k^2+\alpha\mu_k\}$ with the first eigenvalue $\mu_1^2+\alpha\mu_1$, where $\{\mu_k\}$ are the eigenvalues of $-\Delta$ in $L^2(\Omega)$ with the first eigenvalue $\mu_1>0$, so that, if $\alpha>-\mu_1$ then $H_0^1(\Omega)\cap H^2(\Omega)$ is also a Hilbert space with the following inner product%
\begin{equation*}
\langle u,v\rangle_\alpha=\int_{\Omega}(\Delta u\Delta v+\alpha\nabla u\nabla v)dx,%
\end{equation*}
and then one can study $(\mathcal{P}_\alpha)$ by the variational method under some suitable conditions on $g(x,u)$ in the case of $\alpha>-\mu_1$.  However, when $V(x)\neq 0$, the operator $\Delta^2-a_0\Delta+V(x)$ in $\mathcal{V}$ is much more complex due to the potential $V(x)$ and the spectrum of the operator $\Delta^2-a_0\Delta+V(x)$ in $\mathcal{V}$ is not clear in the case of $a_0<0$, also
the variational setting of \eqref{eq1000} is not clear in the case of $a_0<0$.  Due to these reasons, to our best knowledge,
there is few study on Porblem~\eqref{eq1000} for the case of $a_0<0$.  Therefore, a natural question is that does Problem~\eqref{eq1000} have a nontrivial solution for some $a_0<0$ and $V(x)\neq 0$?  The purpose of this paper is to explore this question.%

We assume $V(x)=\lambda b(x)+b_0$, where $b_0\in\bbr$ is a constant, $\lambda>0$ is a parameter and $b(x)$ satisfies the following conditions:%
\begin{enumerate}
\item[$(B_1)$] $b(x)\in C(\bbr^N)$ and $b(x)\geq0$ on $\bbr^N$.%
\item[$(B_2)$] There exists $b_\infty>0$ such that $|\mathcal{B}_\infty|<+\infty$, where $\mathcal{B}_\infty=\{x\in\bbr^N\mid b(x)<b_\infty\}$ and $|\mathcal{B}_\infty|$ is the Lebesgue measure of the set $\mathcal{B}_\infty$.%
\item[$(B_3)$] $\Omega=\text{int} b^{-1}(0)$ is a bounded domain having the smooth boundary $\partial\Omega$ and $\overline{\Omega}=b^{-1}(0)$.%
\end{enumerate}
$\lambda b(x)$ is called as the steep potential well for $\lambda$ sufficiently large under the conditions $(B_1)$--$(B_3)$ and the depth of the well is controlled by the parameter $\lambda$.  Such potentials were first introduced by Bartsch and Wang in \cite{BW95} for the scalar Schr\"odinger equations.  An interesting phenomenon for this kind of Schr\"odinger equations is that, one can expect to find the solutions which are concentrated at the bottom of the wells as the depth goes to infinity.  Due to this interesting property, such topic for the scalar Schr\"odinger equations was studied extensively in the past decade.  We refer the readers to \cite{AFS09,BT13,DT03,DS07,LHL11,ST09,WZ09} and the references therein.  Recently, the steep potential well was also considered for some other elliptic equations and systems, see for example \cite{FSX10,GT121,JZ11,LCW12,SW14,WZ12,YT13,YT14,ZLZ13} and the references therein.  In particular, the steep potential well was introduced to the biharmonic equations in \cite{LCW12} and was further studied in \cite{WZ12,YT13} in the case of $a_0\geq0$.  For the nonlinearity, we assume that $f(x,t)=f(t)$ and satisfies the following conditions:
\begin{enumerate}
\item[$(F_1)$] $\lim_{t\to0}\frac{f(t)}{t}=l_0\geq0$.%
\item[$(F_2)$] There exists $2\leq p<2^*$ such that $\lim_{t\to\infty}\frac{f(t)}{|t|^{p-2}t}=l_\infty>0$, where $2^*=\frac{2N}{N-2}$.%
\item[$(F_3)$] $\frac{f(t)}{|t|}$ is nondecreasing on $\bbr\backslash\{0\}$.%
\item[$(F_4)$] There exists $l_*\in(0, l_\infty]$ such that $f(t)t-2F(t)\geq l_*|t|^{p}$ and $F(t)\geq0$ for all $t\in\bbr$, where $F(t)=\int_0^t f(s)ds$.%
\end{enumerate}
Now, under the conditions $(B_1)$--$(B_3)$ and $(F_1)$--$(F_4)$, we mainly study the following problem in this paper:
$$
\left\{\aligned&\Delta^2u-a_0\Delta u+(\lambda b(x)+b_0)u=f(u)&\text{ in }\bbr^N,\\%
&u\in\h,\endaligned\right.\eqno{(\mathcal{P}_{\lambda})}%
$$

In order to establish a variational framework of $(\mathcal{P}_{\lambda})$ in the case of $a_0<0$, we need study the spectrum and Morse index of the operator $\Delta^2-a_0\Delta+(\lambda b(x)+b_0)$ in a suitable Hilbert space under the conditions $(B_1)$--$(B_3)$.
We will borrow some ideas of \cite{DS07} (see also \cite{ZLZ13}) to carry on this study.  Note that in the case of $a_0<0$, the negative part of the operator $\Delta^2-a_0\Delta+(\lambda b(x)+b_0)$ is generated by not only $(\lambda b(x)+b_0)^-$ but also $-a_0\Delta$, where $(\lambda b(x)+b_0)^-=\max\{-(\lambda b(x)+b_0), 0\}$.  Therefore, some new ideas and modifications are needed in establishing a variational framework of $(\mathcal{P}_{\lambda})$ in the case of $a_0<0$.

Before we state our results, we need to introduce some notations.  Let $\Omega$ be given in the condition $(B_3)$ and let $\{\mu_k\}$ be the eigenvalues of $-\Delta$ in $L^2(\Omega)$, then it is well known that $0<\mu_1<\mu_2\leq\mu_3\leq\cdots\leq\mu_k<\cdots$ with $\mu_k\to+\infty$ as $k\to\infty$ and $\phi_k$ are orthogonal in $L^2(\Omega)\cap H_0^1(\Omega)$ and span$\{\phi_k\}=H_0^1(\Omega)$, where $\phi_k$ are the eigenfunctions of $\mu_k$.  Since $\partial\Omega$ is smooth due to the condition $(B_3)$, it is also well known that $\{\phi_k\}\subset C_0^\infty(\overline{\Omega})$.  Let $H$ be the Hilbert space $H^2(\Omega)\cap H_0^1(\Omega)$ equipped with the inner product
\begin{equation*}
\langle u,v\rangle_{H}=\int_{\Omega}(\Delta u\Delta v+\max\{a_0,0\}\nabla u\nabla v+\max\{b_0,0\}uv)dx.
\end{equation*}
Then span$\{\phi_k\}=H$ and $\phi_k$ are orthogonal in $H$.  We re-denote $\{\mu_k\}$ by $\{\overline{\mu}_n\}$ such that $\overline{\mu}_j<\overline{\mu}_{j+1}$ for all $j\in\bbn$.  Clearly, $\mu_1=\overline{\mu}_1$.  In the case of $\min\{a_0,b_0\}<0$, we denote
\begin{equation}\label{eq9998}
k_0^*=\inf\bigg\{k\in\bbn: \frac{\overline{\mu}_k^2+\max\{a_0, 0\}\overline{\mu}_k+\max\{b_0,0\}}{\max\{-a_0, 0\}\overline{\mu}_k+\max\{-b_0,0\}}>1\bigg\}
\end{equation}
and $\overline{\mu}_0=0$.  Then the main results obtained in this paper can be stated as follows.
\begin{theorem}\label{thm0001}
(The superlinear case) Suppose that the conditions $(B_1)$--$(B_3)$, $(F_1)$--$(F_2)$ and $(F_4)$ hold with $p>2$.  If either $\min\{a_0,b_0\}\geq0$ or $\min\{a_0,b_0\}<0$ with
\begin{equation*}
\frac{\overline{\mu}_{k_0^*-1}^2+\max\{a_0, 0\}\overline{\mu}_{k_0^*-1}+\max\{b_0,0\}}{\max\{-a_0, 0\}\overline{\mu}_{k_0^*-1}+\max\{-b_0,0\}}<1
\end{equation*}
then there exist positive constants $\overline{l}_0$ and $\widehat{\Lambda}$ such that Problem~$(\mathcal{P}_\lambda)$ has a nontrivial solution for $\lambda>\widehat{\Lambda}$ in the case of $l_0<\overline{l}_0$.
\end{theorem}

\begin{theorem}\label{thm0002}
(The asymptotically linear case) Suppose that the conditions $(B_1)$--$(B_3)$ and $(F_1)$--$(F_3)$ hold with $p=2$.  If either $\min\{a_0,b_0\}\geq0$ or $\min\{a_0,b_0\}<0$ with
\begin{equation*}
\frac{\overline{\mu}_{k_0^*-1}^2+\max\{a_0, 0\}\overline{\mu}_{k_0^*-1}+\max\{b_0,0\}}{\max\{-a_0, 0\}\overline{\mu}_{k_0^*-1}+\max\{-b_0,0\}}<1
\end{equation*}
then there exist positive constants $\overline{l}_0$, $\overline{l}_\infty$ and $\widehat{\Lambda}$ such that Problem~$(\mathcal{P}_\lambda)$ has a nontrivial solution for $\lambda>\widehat{\Lambda}$ in the cases of $l_0<\overline{l}_0$ and $l_\infty>\overline{l}_\infty$ with $l_\infty\not\in\sigma(\Delta^2-a_0\Delta+b_0, L^2(\Omega))$, where $\sigma(\Delta^2-a_0\Delta+b_0, L^2(\Omega))$ is the spectrum of $\Delta^2-a_0\Delta+b_0$ in $L^2(\Omega)$.
\end{theorem}

Since Problem~$(\mathcal{P}_\lambda)$  depends on the parameter $\lambda$, it is natural to investigate the concentration behavior of the solutions for $\lambda\to+\infty$.  Our result in this topic can be stated as follows.
\begin{theorem}\label{thm0003}
Suppose that $u_\lambda$ is the nontrivial solution of Problem~$(\mathcal{P}_{\lambda})$ obtained by Theorem~\ref{thm0001} or Theorem~\ref{thm0002}.  Then $u_\lambda\to u_*$ strongly in $\h$ as $\lambda\to+\infty$ up to a subsequence for some $u_*\in H_0^1(\Omega)\cap H^2(\Omega)$.  Furthermore, $u_*$ is a nontrivial weak solution of the following equation
\begin{equation}\label{eq0030}
\left\{\aligned&\Delta^2u-a_0\Delta u+b_0u=f(u)&\quad\text{in }\Omega,\\%
&u=\Delta u=0&\quad\text{on }\partial\Omega.\endaligned\right.%
\end{equation}
\end{theorem}
\begin{remark}{\em
Theorem~\ref{thm0003} actually gives an existence result to \eqref{eq0030} in the case of $a_0\leq-\mu_1$, where the nonlinearities are superlinear and subcritical or asymptotically linear at infinity.  To our best knowledge, such result has not been obtained in literatures no matter what the nonlinearities are superlinear and subcritical or asymptotically linear at infinity.}
\end{remark}

Through this paper, $C$ and $C'$ will be indiscriminately used to denote various positive constants, $o_n(1)$ and $o_\lambda(1)$ will always denote the quantities tending towards zero as $n\to\infty$ and $\lambda\to+\infty$ respecitvely.%

\section{The variational setting}
In this section, we will give the variational setting of $(\mathcal{P}_\lambda)$.  Let%
\begin{equation*}
E_\lambda=\{u\in\h:\int_{\bbr^N}(\lambda b(x)+b_0)^+u^2dx<+\infty\},%
\end{equation*}
where $(\lambda b(x)+b_0)^+=\max\{\lambda b(x)+b_0, 0\}$.  Then by the condition $(B_1)$, $E_\lambda$ is a Hilbert space equipped with the inner product%
\begin{equation*}
\langle u,v \rangle_\lambda=\int_{\bbr^N}(\Delta u\Delta v+\max\{a_0, 0\}\nabla u\nabla v+(\lambda b(x)+b_0)^+uv)dx%
\end{equation*}
for all $\lambda>0$.  The corresponding norm in $E_\lambda$ is given by $\|u\|_\lambda=\langle u,u \rangle_\lambda^{\frac12}$.  By the condition $(B_2)$ and the Sobolev inequality and the H\"older inequality, for $\lambda>\max\{0, -\frac{b_0}{b_\infty}\}$, we have %
\begin{eqnarray}
\|u\|_{L^2(\bbr^N)}^2\leq|\mathcal{B}_\infty|^{\frac{2^*-2}{2^*}}S^{-1}\|\nabla u\|_{L^2(\bbr^N)}^2+\frac{1}{\lambda b_\infty+b_0}\|u\|_{\lambda}^2\label{eq0019},%
\end{eqnarray}
where $\|\cdot\|_{L^p(\bbr^N)}$ is the usual norms in $L^p(\bbr^N)$ for all $p\geq1$, $S$ is the best Sobolev embedding constants from $D^{1,2}(\bbr^N)$ to $L^{2^*}(\bbr^N)$ and given by%
\begin{equation*}
S=\inf\{\|\nabla u\|_{L^2(\bbr^N)}^2 : u\in D^{1,2}(\bbr^N), \|u\|_{L^{2^*}(\bbr^N)}^2=1\}.%
\end{equation*}
Let $\mathcal{A}_\infty=|\mathcal{B}_\infty|^{\frac{2^*-2}{2^*}}S^{-1}$.  If $a_0>0$ then by \eqref{eq0019}, we can see
\begin{eqnarray}
\|u\|_{L^2(\bbr^N)}^2\leq(\mathcal{A}_\infty a_0^{-1}+\frac{1}{\lambda b_\infty+b_0})\|u\|_\lambda^2\label{eq0001}.
\end{eqnarray}
If $a_0\leq0$ then by \eqref{eq0019} and the Young and the Gagliardo-Nirenberg inequalities, we have
\begin{equation}  \label{eq0009}
\|u\|_{L^2(\bbr^N)}^2\leq(4\mathcal{A}_\infty^2 B^4_0+\frac{2}{\lambda b_\infty+b_0})\|u\|_{\lambda}^2,
\end{equation}
where $B_0>0$ is the constant in the Gagliardo-Nirenberg inequality:
\begin{eqnarray*}
\|\nabla u\|_{L^2(\bbr^N)}\leq B_0\|\Delta u\|_{L^2(\bbr^N)}^{\frac12}\|u\|_{L^2(\bbr^N)}^{\frac12}\quad\text{for all }u\in H^2(\bbr^N).
\end{eqnarray*}
Thus, by \eqref{eq0001} and \eqref{eq0009}, we get
\begin{equation}\label{eq0020}
\|u\|_{L^2(\bbr^N)}^2\leq C_\lambda\|u\|_\lambda^2,
\end{equation}
where
\begin{equation*}
C_\lambda=\left\{\aligned&\mathcal{A}_\infty a_0^{-1}+\frac{1}{\lambda b_\infty+b_0},\quad\text{if }a_0>0,\\
&4\mathcal{A}_\infty^2 B^4_0+\frac{2}{\lambda b_\infty+b_0},\quad\text{if }a_0\leq0.\endaligned\right.
\end{equation*}
\eqref{eq0020}, together with the Gagliardo-Nirenberg inequality, implies
\begin{eqnarray}
\|\nabla u\|_{L^2(\bbr^N)}\leq B_0\|\Delta u\|_{L^2(\bbr^N)}^{\frac12}\|u\|_{L^2(\bbr^N)}^{\frac12}
\leq B_0C_\lambda^{\frac14}\|u\|_{\lambda}.\label{eq0010}
\end{eqnarray}
Combining  \eqref{eq0020} and \eqref{eq0010}, we deduce that $E_\lambda$ is embedded continuously into $\h$ for $\lambda>\max\{0, \frac{-b_0}{b_\infty}\}$.  On the other hand, by \eqref{eq0020} and the conditions $(B_1)$--$(B_2)$, we also have
\begin{eqnarray}
\int_{\bbr^N}(\lambda b(x)+b_0)^-u^2dx\leq\max\{0, -b_0\}\|u\|_{L^2(\bbr^N)}^2
\leq\max\{0, -b_0\}C_\lambda^{\frac12}\|u\|_{\lambda}^2.\label{eq0011}
\end{eqnarray}
Using $(F_1)$--$(F_3)$, we can obtain that $0\leq F(t)\leq2l_0|t|^2+C|t|^p$.  It follows from \eqref{eq0020}--\eqref{eq0011} and the Sobolev embedding theorem that the functional $\mathcal{E}_\lambda(u):E_\lambda\to \mathbb{R}$ given by
\begin{equation*}\label{eq0012}
\mathcal{E}_{\lambda}(u)=\frac12\mathcal{D}_{\lambda}(u,u)-\int_{\bbr^N}F(x,u)dx%
\end{equation*}
is well defined and it belongs to $C^1$ for $\lambda>\max\{0, -\frac{b_0}{b_\infty}\}$, where
\begin{equation*}
\mathcal{D}_\lambda(u,v)=\langle u,v\rangle_{\lambda}-\mathcal{G}_\lambda(u,v)
\end{equation*}
with $\mathcal{G}_{\lambda}(u,v)=\max\{-a_0,0\}\langle \nabla u,\nabla v\rangle_{L^2(\bbr^N)}+\int_{\bbr^N}(\lambda b(x)+b_0)^-uvdx$.
Furthermore, by using a standard argument and the conditions $(F_1)$--$(F_2)$, we can show that $\mathcal{E}_\lambda(u)$ is the corresponding functional of $(\mathcal{P}_\lambda)$.   In what follows, we will make some further observations on $\mathcal{D}_{\lambda}(u,u)$.

If $\min\{a_0,b_0\}\geq0$ then $\mathcal{G}_{\lambda}(u,v)=0$, which gives that $\mathcal{D}_\lambda(u,v)=\langle u,v\rangle_{\lambda}$ for all $(u,v)\in E_\lambda$ and then $\mathcal{D}_\lambda(u,u)$ is definite on $E_\lambda$.

If $\min\{a_0,b_0\}<0$ and let
\begin{equation*}
\mathcal{M}_0=\bigg\{u\in H:\int_{\Omega}(\max\{-a_0, 0\}|\nabla u|^2+\max\{-b_0,0\}|u|^2)dx=1\bigg\}
\end{equation*}
and
\begin{equation}\label{eq01}
\mathcal{N}_j=\text{span}\bigg\{\phi_i:\phi_i\text{ is the corresponding function of } \overline{\mu}_j\bigg\},
\end{equation}
then it is well known that $\mathcal{M}_0$ is a natural constraint in $H$ and dim$(\mathcal{N}_j)<+\infty$ for all $j\in\bbn$.  In particular, dim$(\mathcal{N}_1)=1$ and $\phi_1$ is positive on $\Omega$.  Moreover,
 let
\begin{equation}\label{eq02}
\beta_j^0=\inf_{(\mathcal{M}^{j-1})^{\perp}}\int_{\Omega}(|\Delta u|^2+\max\{a_0,0\}|\nabla u|^2+\max\{b_0,0\}|u|^2)dx,\quad j=1,2,\cdots,
\end{equation}
where $(\mathcal{M}^{j-1})^{\perp}=\{u\in\mathcal{M}_0:\langle u,v\rangle_{H}=0\text{ for all }v\in\bigoplus_{i=1}^{j-1}\mathcal{N}_i\}$,
then by span$\{\phi_k\}=H$ and the orthogonality of $\{\phi_k\}$  in $H$, we can easily see from the Sobolev embedding theorem, the Gagliardo-Nirenberg inequality and the method of Lagrange multipliers that
\begin{equation*}
\beta_j^0=\frac{\overline{\mu}_j^2+\max\{a_0, 0\}\overline{\mu}_j+\max\{b_0,0\}}{\max\{-a_0, 0\}\overline{\mu}_j+\max\{-b_0,0\}}\quad\text{for all }j\in\bbn
\end{equation*}
and $\beta_j^0$ can be attained by $u\in H$ if and only if $u\in\mathcal{N}_j\cap \mathcal{M}_0$.
\begin{lemma}\label{lem0001}
Suppose that the conditions $(B_1)$--$(B_3)$ hold.  If
\begin{equation*}
\lambda>\Lambda_1:=\frac{(\max\{-a_0, 0\}\beta_1^0)^2B_0^4-b_0}{b_\infty},
\end{equation*}
then $\beta_1(\lambda)=\inf_{\mathcal{M}_{\lambda}}\|u\|^2_{\lambda}$ can be attained by some $e_1(\lambda)\in \mathcal{M}_{\lambda}$, where $\mathcal{M}_{\lambda}=\{u\in E_{\lambda}: \mathcal{G}_{\lambda}(u,u)=1\}$.  Moreover, $(e_1(\lambda),\beta_1(\lambda))$ satisfies the following equation
\begin{equation}\label{eq0022}
\left\{\aligned&\Delta^2 u-\max\{a_0, 0\}\Delta u+(\lambda b(x)+b_0)^+u=\beta(\max\{-a_0, 0\}\Delta u+(\lambda b(x)+b_0)^-u),\quad\text{in }\bbr^N,\\
&u\in\h,\endaligned\right.
\end{equation}
and $(e_1(\lambda),\beta_1(\lambda))\to(\phi_1,\beta_1^0)$ strongly in $\h$ as $\lambda\to+\infty$ up to a subsequence.
\end{lemma}
\begin{proof}
We first prove that $\beta_1(\lambda)$ can be attained for $\lambda>\Lambda_1$.  Indeed,
by the Ekeland principle, there exists $\{u_n\}\subset\mathcal{M}_{\lambda}$ such that
\begin{enumerate}
\item[$(1)$] $\|u_n\|^2_{\lambda}=\beta_1(\lambda)+o_n(1)$;
\item[$(2)$] $\|v\|^2_{\lambda}\geq\|u_n\|^2_{\lambda}-\frac1n\|v-u_n\|_{\lambda}$ for all $v\in\mathcal{M}_{\lambda}$.
\end{enumerate}
For each $n\in\mathbb{N}$ and $w\in E_{\lambda}$,  by applying the implicit function theorem in a standard way and noting that  the conditions $(B_1)$--$(B_2)$, we can see that there exist $\ve_n>0$ and $t_n(l)\in C^1([-\ve_n, \ve_n])$ with $t_n(0)=1$ and $t_n'(0)=-\mathcal{G}_\lambda(u_n,w)$ such that $t_n(l)u_n+lw\in\mathcal{M}_{\lambda}$ for $l\in(0, \ve_n]$.  It follows from $(2)$ that
\begin{eqnarray*}
0&\geq&\|u_n\|^2_{\lambda}-\|t_n(l)u_n+lw\|^2_{\lambda}-\frac1n\|(t_n(l)-1)u_n+lw\|_{\lambda}\notag\\
&\geq&(1-t_n(l)^2)\|u_n\|^2_{\lambda}-2t_n(l)l\langle u_n,w\rangle_{\lambda}-l^2\|w\|^2_{\lambda}-\frac{t_n(l)-1}{n}\|u_n\|_{\lambda}-\frac ln\|w\|_{\lambda}.
\end{eqnarray*}
Multiplying this inequality with $l^{-1}$ on both side and letting $l\to0^+$, we deduce
\begin{eqnarray*}
0&\geq&(2\|u_n\|^2_{\lambda}+\frac1n\|u_n\|_{\lambda})\int_{\bbr^N}(\max\{-a_0, 0\}\nabla u_n\nabla w+(\lambda b(x)+b_0)^-u_nw)dx\\
&&-2\langle u_n,w\rangle_{\lambda}-\frac1n\|w\|_{\lambda}\\
&=&2(\beta_1(\lambda)\mathcal{G}_\lambda(u_n,w)-\langle u_n,w\rangle_{\lambda})+o_n(1).
\end{eqnarray*}
Since $w\in E_{\lambda}$ is arbitrary, we must have
\begin{equation}\label{eq0013}
o_n(1)=\beta_1(\lambda)\mathcal{G}_\lambda(u_n,w)-\langle u_n,w\rangle_{\lambda}
\end{equation}
for all $w\in E_{\lambda}$.  Let $J_{\lambda}(u)=\frac12\|u\|^2_{\lambda}-\frac{\beta_1(\lambda)}2\mathcal{G}_\lambda(u,u)$.  Then $J_\lambda'(u_n)w=o_n(1)$ for all $w\in E_{\lambda}$.  In particular, by the choice of $\{u_n\}$, we can also see that $\langle J_\lambda'(u_n),u_n\rangle_{E_{\lambda}^*, E_{\lambda}}=o_n(1)$, where $E_\lambda^*$ is the dual space of $E_\lambda$ and $\langle \cdot,\cdot\rangle_{E_{\lambda}^*, E_{\lambda}}$ is the duality pairing of $E_{\lambda}^*$ and $E_{\lambda}$.  On the other hand, by $(1)$, $\{u_n\}$ is bounded in $E_{\lambda}$.  Thus, without loss of generality, we may assume that $u_n\rightharpoonup e_1(\lambda)$ weakly in $E_{\lambda}$ as $n\to\infty$.  Note that $J_\lambda(u)$ is $C^2$ in $E_{\lambda}$ for $\lambda>\max\{0, \frac{-b_0}{b_\infty}\}$ due to the conditions $(B_1)$--$(B_2)$,  we have $J_\lambda'(e_1(\lambda))=0$ in $E_{\lambda}^*$.  It follows from the Sobolev embedding theorem, the conditions $(B_1)$--$(B_2)$ and similar arguments used in the proofs of \eqref{eq0020} and \eqref{eq0010} that
\begin{eqnarray}
o_n(1)&=&\langle J_\lambda'(u_n)-J_\lambda'(e_1(\lambda)),u_n-e_1(\lambda)\rangle_{E_{\lambda}^*, E_{\lambda}}\notag\\
&=&\|u_n-e_1(\lambda)\|_{\lambda}^2-\beta_1(\lambda)\bigg(\mathcal{G}_\lambda(u_n-e_1(\lambda),u_n-e_1(\lambda))\bigg)\notag\\
&\geq&\bigg(1-\frac{\beta_1(\lambda)\max\{-a_0, 0\}B_0^2}{(\lambda b_\infty+b_0)^{\frac12}}\bigg)\|u_n-e_1(\lambda)\|_{\lambda}^2+o_n(1).\label{eq0014}
\end{eqnarray}
Note that by the condition $(B_3)$, we get
\begin{equation}\label{eq0016}
\beta_1(\lambda)\leq\frac{\|\phi_1\|_{\lambda}^2}{\mathcal{G}_\lambda(\phi_1,\phi_1)}
=\beta_1^0.
\end{equation}
It follows from \eqref{eq0014} that for $\lambda>\Lambda_1$, $u_n\to e_1(\lambda)$ strongly in $E_{\lambda}$ as $n\to\infty$.  Thus, $e_1(\lambda)\in\mathcal{M}_\lambda$ and $\beta_1(\lambda)$ can be attained by $e_1(\lambda)$ for $\lambda>\Lambda_1$.  By \eqref{eq0013}, we also see that $(e_1(\lambda), \beta_1(\lambda))$ satisfies \eqref{eq0022}.

To complete proof of this lemma, we shall show that $(e_1(\lambda),\beta_1(\lambda))\to(\phi_1,\beta_1^0)$ strongly in $\h$ as $\lambda\to+\infty$ up to a subsequence.  Indeed, by \eqref{eq0016}, we know that $\|e_1(\lambda)\|_\lambda\leq\beta_1^0$ for all $\lambda>\Lambda_1$, which, together with \eqref{eq0020} and \eqref{eq0010}, implies that $\{e_1(\lambda)\}$ is bounded in $\h$.  Without loss of generality, we may assume that $e_1(\lambda)\rightharpoonup e_1^*$ weakly in $\h$ as $\lambda\to+\infty$.  Since $\|e_1(\lambda)\|_\lambda\leq\beta_1^0$ for all $\lambda>\Lambda_1$, by the condition $(B_1)$ and the Fatou lemma, we have
\begin{equation}\label{eq0007}
0=\liminf_{\lambda\to+\infty}\frac{\beta_{1}^0}{\lambda}\geq\liminf_{\lambda\to+\infty}\int_{\bbr^N}(b(x)+\frac{b_0}{\lambda})^+[e_1(\lambda)]^2dx
\geq\int_{\bbr^N}b(x)[e_1^*]^2dx.
\end{equation}
It follows from the condition $(B_3)$ that $e_1^*\in H_0^1(\Omega)$.  Thanks to the condition $(B_2)$ and the fact that $E_\lambda$ is embedded continuously into $\h$ for $\lambda>\max\{0, \frac{-b_0}{b_\infty}\}$, we can see from \eqref{eq0007} and the Sobolev embedding theorem that $e_1(\lambda)\to e_1^*$ strongly in $L^2(\bbr^N)$ as $\lambda\to+\infty$, then the Gagliardo-Nirenberg inequality yields that $e_1(\lambda)\to e_1^*$ strongly in $H^1(\bbr^N)$ as $\lambda\to+\infty$.  On the other hand, by the conditions $(B_1)$--$(B_3)$, we obtain from a variant of the
Lebesgue dominated convergence theorem (cf. \cite[Theorem~2.2]{PK74}) and the fact that $e_1(\lambda)\to e_1^*$ strongly in $H^1(\bbr^N)$ as $\lambda\to+\infty$ that $e_1^*\in\mathcal{M}_0$, which, together with the definition of $\beta_1^0$ and the conditions $(B_1)$--$(B_3)$, deduces
\begin{eqnarray*}
\liminf_{\lambda\to+\infty}\beta_1(\lambda)\geq\int_{\Omega}\big(|\Delta e_1^*|^2+\max\{a_0, 0\}|\nabla e_1^*|^2+\max\{b_0,0\}|e_1^*|^2\big)dx\geq\beta_1^0.
\end{eqnarray*}
Thus, we must have $\lim_{\lambda\to+\infty}\beta_1(\lambda)=\beta_1^0$.  Now, thanks to the conditions $(B_1)$--$(B_3)$ once more, we can see from the weak convergence of $\{e_1(\lambda)\}$ in $\h$  and the Fatou lemma that
\begin{eqnarray}
\beta_1^0=\lim_{\lambda\to+\infty}\|e_1(\lambda)\|_{\lambda}^2
\geq\int_{\Omega}|\Delta e_1^*|^2+\max\{a_0, 0\}|\nabla e_1^*|^2+\max\{b_0,0\}|e_1^*|^2dx\geq\beta_1^0.\label{eq0021}
\end{eqnarray}
Thus $\Delta e_k(\lambda)\to \Delta e_k^*$ strongly in $L^2(\bbr^N)$ as $\lambda\to+\infty$ since
$e_1(\lambda)\to e_1^*$ strongly in $L^2(\bbr^N)\cap H^1(\bbr^N)$ as $\lambda\to+\infty$, so that $e_1(\lambda)\to e_1^*$ strongly in $\h$ as $\lambda\to+\infty$.  Note that $e_1^*\in H$ which  attains $\beta_1^0$, we finally have $e_1^*=\phi_1$.
\end{proof}

Let $\mathcal{N}_{\lambda,1}=\text{span}\{u\in \mathcal{M}_\lambda:\|u\|_\lambda^2=\beta_1(\lambda)\}$.  Then by Lemma~\ref{lem0001}, we can see that $e_1(\lambda)\in\mathcal{N}_{\lambda,1}$ for $\lambda>\Lambda_1$.  Moreover, we also have the following lemma.
\begin{lemma}\label{lem0002}
Suppose that the conditions $(B_1)$--$(B_3)$ hold.  Then there exists $\Lambda_1^*\geq\Lambda_1$ such that $\mathcal{N}_{\lambda,1}=\text{span}\{e_1(\lambda)\}$ for $\lambda>\Lambda_1^*$.
\end{lemma}
\begin{proof}
Suppose on the contrary that there exists $\{\lambda_n\}$ satisfying $\lambda_n\to+\infty$ as $n\to\infty$ such that $\mathcal{N}_{\lambda_n,1}\not=\text{span}\{e_1(\lambda_n)\}\subset \h$.  It follows that there exists $u(\lambda_n)\in\mathcal{N}_{\lambda_n,1}$ satisfying $u(\lambda_n)\not\in\text{span}\{e_1(\lambda_n)\}$ for all $n\in\bbn$.  Without loss of generality, we may assume that $\langle u(\lambda_n), e_1(\lambda_n)\rangle_{\lambda_n}=0$ for all $n\in\bbn$.  Similarity as in the proof of Lemma~\ref{lem0001}, going if necessary to a subsequence, we may get that $u(\lambda_n)=\phi_1+o_n(1)$ strongly in $\h$ and $\int_{\bbr^N}(\lambda_nb(x)+b_0)^+[u(\lambda_n)]^2dx=\int_{\Omega}\max\{b_0, 0\}\phi_1^2dx+o_n(1)$, which, together with  Lemma~\ref{lem0001}, implies that $\sqrt{(\lambda_nb(x)+b_0)^+}(u(\lambda_n)-e_1(\lambda_n))=o_n(1)$ in $L^2(\bbr^N)$.  It follows from a variant of the
Lebesgue dominated convergence theorem (cf. \cite[Theorem~2.2]{PK74}) that
\begin{equation*}
\|e_1(\lambda_n)-u(\lambda_n)\|_{\lambda_n}^2=o_n(1).
\end{equation*}
Therefore, we have
\begin{equation*}
0<2\beta_1^0=\lim_{n\to\infty}(\|e_1(\lambda_n)\|_{\lambda_n}^2+\|u(\lambda_n)\|_{\lambda_n}^2)
=\lim_{n\to\infty}\|e_1(\lambda_n)-u(\lambda_n)\|_{\lambda_n}^2=0,
\end{equation*}
this is a contradiction.
\end{proof}

Let
\begin{equation*}
\beta_2(\lambda)=\inf_{(\mathcal{M}_{\lambda}^{1})^{\perp}}\|u\|^2_{\lambda},
\end{equation*}
where $(\mathcal{M}_{\lambda}^1)^{\perp}=\{u\in\mathcal{M}_{\lambda}:\langle u,v\rangle_\lambda=0\text{ for all }v\in\mathcal{N}_{\lambda,1}\}$.  Then it is easy to see that $\beta_2(\lambda)\geq\beta_1(\lambda)$ for $\lambda>\Lambda_1$.  Thus, $\beta_2(\lambda)$ is well defined.  Furthermore, we have the following lemma.
\begin{lemma}\label{lem0003}
Suppose that the conditions $(B_1)$--$(B_3)$ hold and let $\beta_2^0$ be given in \eqref{eq02}.  If
\begin{equation*}
\lambda>\Lambda_2:=\frac{(\max\{-a_0, 0\}\beta_2^0)^2B_0^4-b_0}{b_\infty},
\end{equation*}
then $\beta_2(\lambda)$ can be attained for some $e_2(\lambda)\in E_{\lambda}$.  Moreover, $(e_2(\lambda), \beta_2(\lambda))$ satisfies \eqref{eq0022} and $(e_2(\lambda),\beta_2(\lambda))\to(e_2^0,\beta_2^0)$ strongly in $\h$ as $\lambda\to+\infty$ up to a subsequence for some $e_2^0\in\mathcal{N}_2$, where $\mathcal{N}_2$ is defined in \eqref{eq01}.
\end{lemma}
\begin{proof}
For the sake of clarity, the proof will be performed through the following three steps.

{\bf Step.~1}\quad We prove that $\limsup_{\lambda\to+\infty}\beta_2(\lambda)\leq\beta_2^0$.

Let $\varphi\in\mathcal{N}_2$.  Then $\varphi=\varphi^-_\lambda+\varphi_\lambda^+$, where $\varphi^-_\lambda$ and $\varphi_\lambda^+$ are the projections of $\varphi$ in $\mathcal{N}_{\lambda,1}$ and $(\mathcal{M}_{\lambda}^{1})^{\perp}$ respectively.  It follows from $\mathcal{N}_2\subset(\mathcal{M}^{1})^{\perp}$ and Lemmas~\ref{lem0001} and \ref{lem0002} that $\lim_{\lambda\to+\infty}\|\varphi^-_\lambda\|_{\lambda}^2=\lim_{\lambda\to+\infty}\langle \varphi^-_\lambda, \varphi\rangle_{\lambda}=0$ up to a subsequence. By \eqref{eq0010}--\eqref{eq0011}, $\lim_{\lambda\to+\infty}\mathcal{G}_\lambda(\varphi^-_\lambda,\varphi^-_\lambda)=0$ up to a subsequence.
Now, using the definitions of $\beta_2(\lambda)$ and $\beta_2^0$, the conditions $(B_1)$--$(B_3)$ and Lemmas~\ref{lem0001}--\ref{lem0002}, we have
\begin{eqnarray*}
\limsup_{\lambda\to+\infty}\beta_2(\lambda)&\leq&\limsup_{\lambda\to+\infty}
\frac{\|\varphi_\lambda^+\|_{\lambda}^2}{\mathcal{G}_\lambda(\varphi_\lambda^+,\varphi_\lambda^+)}\\
&=&\limsup_{\lambda\to+\infty}
\frac{\|\varphi-\varphi_\lambda^-\|_{\lambda}^2}
{\mathcal{G}_\lambda(\varphi-\varphi_\lambda^-,\varphi-\varphi_\lambda^-)}\\
&=&\limsup_{\lambda\to+\infty}
\frac{\|\varphi\|_{\lambda}^2-2\langle\varphi, \varphi_\lambda^-\rangle_{\lambda}+\|\varphi_\lambda^-\|_{\lambda}^2}
{\mathcal{G}_\lambda(\varphi, \varphi)-2\mathcal{G}_\lambda(\varphi,\varphi_\lambda^-)+\mathcal{G}_\lambda(\varphi_\lambda^-,\varphi_\lambda^-)}\\
&=&\frac{\int_{\Omega}|\Delta \varphi|^2+\max\{a_0, 0\}|\nabla \varphi|^2+\max\{b_0,0\}|\varphi|^2dx}{\int_{\Omega}\max\{-a_0,0\}|\nabla \varphi|^2+\max\{-b_0,0\}|\varphi|^2dx}\\
&=&\beta_2^0.
\end{eqnarray*}

{\bf Step.~2}\quad We prove that for $\lambda>\Lambda_2$, $\beta_2(\lambda)$ can be attained by some $e_2(\lambda)\in\h$ and $(e_2(\lambda), \beta_2(\lambda))$ satisfies \eqref{eq0022}.

Indeed, by a similar argument  used in the proof of Lemma~\ref{lem0001}, we can show that there exists $\{u_n\}\subset(\mathcal{M}^{1}_\lambda)^{\perp}$ such that
\begin{enumerate}
\item[$(1)$] $\|u_n\|_{\lambda}^2=\beta_2(\lambda)+o_n(1)$;
\item[$(2)$] $o_n(1)=\beta_2(\lambda)\mathcal{G}_\lambda(u_n,w)-\langle u_n,w\rangle_{\lambda}$ for all $w\in(\mathcal{M}^{1}_\lambda)^{\perp}$.
\end{enumerate}
Clearly, $(1)$ gives the boundedness of $\{u_n\}$ in $E_{\lambda}$, hence we may assume that $u_n\rightharpoonup e_2(\lambda)$ weakly in $E_{\lambda}$ as $n\to\infty$. Since $\{u_n\}\subset(\mathcal{M}^{1}_\lambda)^{\perp}$, we must have $e_2(\lambda)\in(\mathcal{M}^{1}_\lambda)^{\perp}$.  By $(2)$, we obtain
\begin{equation*}
0=\beta_2(\lambda)\mathcal{G}_\lambda(e_2(\lambda),e_2(\lambda))-\|e_2(\lambda)\|_{\lambda}^2,
\end{equation*}
which, together with $(1)$ and similar arguments  in the proof of  \eqref{eq0014}, implies
\begin{eqnarray*}
o_n(1)\geq\bigg(1-\frac{\beta_2(\lambda)\max\{-a_0, 0\}B_0^2}{(\lambda b_\infty+b_0)^{\frac12}}\bigg)\|u_n-u_k(\lambda)\|_{\lambda}^2+o_n(1).
\end{eqnarray*}
It follows from Step.~1 that $u_n\to e_2(\lambda)$ strongly in $E_{\lambda}$ as $n\to\infty$ for $\lambda>\Lambda_2$.  Hence, by $(1)$--$(2)$, $\beta_2(\lambda)$ is attained by $e_2(\lambda)$ for $\lambda>\Lambda_2$ and $(e_2(\lambda), \beta_2(\lambda))$ satisfies \eqref{eq0022}.

{\bf Step.~3}\quad We prove that $(u_2(\lambda),\beta_2(\lambda))\to(e_2^0,\beta_2^0)$ strongly in $\h$ as $\lambda\to+\infty$ up to a subsequence for some $e_2^0\in\mathcal{N}_2$.

Indeed, similarly  as in the proof of Lemma~\ref{lem0001}, we can show that $(u_2(\lambda),\beta_2(\lambda))\to(e_2^0,\beta_2^0)$ strongly in $H^1(\bbr^N)\cap L^2(\bbr^N)$ as $\lambda\to+\infty$ up to a subsequence for some $e_2^0\in H$.  It follows from the conditions $(B_1)$--$(B_3)$ and a variant of the
Lebesgue dominated convergence theorem (cf. \cite[Theorem~2.2]{PK74}) that $e_2^0\in\mathcal{M}_0$.
By Step.~1, we see that either $e_2^0\in\mathcal{N}_1\cap \mathcal{M}_0$ or $e_2^0\in\mathcal{N}_2\cap \mathcal{M}_0$.  If $e_2^0\in\mathcal{N}_1\cap \mathcal{M}_0$ then by Lemma~\ref{lem0001}, \eqref{eq0021} and the H\"older inequality, we have
\begin{equation*}
0=\lim_{\lambda\to+\infty}\langle u_2(\lambda), e_1(\lambda)\rangle_\lambda=\int_{\Omega}\big(|\Delta \phi_1|^2+\max\{a_0,0\}|\nabla \phi_1|^2+\max\{b_0,0\}|\phi_1|^2\big)dx.
\end{equation*}
It is impossible.  Thus, we must have $e_2^0\in\mathcal{N}_2\cap \mathcal{M}_0$, which implies that $\liminf_{\lambda\to+\infty}\beta_2(\lambda)\geq\beta_2^0$, this, together with
Step.~1, yields that $\lim_{\lambda\to+\infty}\beta_2(\lambda)=\beta_2^0$ and $e_2^0$ attains $\beta_2^0$.  Now, by a similar argument used in the proof of Lemma~\ref{lem0001}, we get that that $u_2(\lambda)\to e_2^0$ strongly in $\h$ as $\lambda\to+\infty$ up to a subsequence.
\end{proof}

Let $\mathcal{N}_{\lambda,2}=\text{span}\{u\in \mathcal{M}_\lambda:\|u\|_\lambda^2=\beta_2(\lambda)\}$.  Then by Lemma~\ref{lem0003}, we can see that $e_2(\lambda)\in\mathcal{N}_{\lambda,2}$ for $\lambda>\Lambda_2$.  Moreover, we also have the following.
\begin{lemma}\label{lem0004}
Suppose that the conditions $(B_1)$--$(B_3)$ hold.  Then there exists $\Lambda_2^*\geq\Lambda_2$ such that $\beta_1(\lambda)<\beta_2(\lambda)$, $\mathcal{N}_{\lambda,1}\perp\mathcal{N}_{\lambda,2}$ and dim$(\mathcal{N}_{\lambda,2})\leq$dim$(\mathcal{N}_2)$ for $\lambda>\Lambda_2^*$.  Here, we say $\mathcal{N}_{\lambda,1}\perp\mathcal{N}_{\lambda,2}$ in the sense that $\langle u, v\rangle_\lambda=0$ for all $u\in\mathcal{N}_{\lambda,1}$ and $v\in\mathcal{N}_{\lambda,2}$.
\end{lemma}
\begin{proof}
Let $u(\lambda), v(\lambda)\in \mathcal{N}_{\lambda,2}$ with $u(\lambda)\not\in\text{span}\{v(\lambda)\}$.  Without loss of generality, we may assume that $\langle u(\lambda), v(\lambda)\rangle_\lambda=0$.  By Lemma~\ref{lem0003}, we can see that $u(\lambda)\to e'$ and $v(\lambda)\to e''$ strongly in $\h$ as $\lambda\to+\infty$ up to a subsequence for some $e',e''\in\mathcal{N}_2$.  If $e'=e''$, then by a similar argument used in the proof of Lemma~\ref{lem0002}, we can show that $0<\beta_2^0=0$, which is a contradiction.  Hence, we must have $\langle e', e''\rangle_{H}=0$ since $\langle u(\lambda), v(\lambda)\rangle_\lambda=0$.  It follows from Lemmas~\ref{lem0001} and \ref{lem0003} that there exists $\Lambda_2^*\geq\Lambda_2$ such that $\beta_1(\lambda)<\beta_2(\lambda)$, $\mathcal{N}_{\lambda,1}\perp\mathcal{N}_{\lambda,2}$ and dim$(\mathcal{N}_{\lambda,2})\leq$dim$(\mathcal{N}_2)$ for $\lambda>\Lambda_2^*$.
\end{proof}

Now, define $\beta_k(\lambda)$ as
\begin{equation*}
\beta_k(\lambda)=\inf_{(\mathcal{M}_{\lambda}^{k-1})^{\perp}}\|u\|^2_{\lambda}\quad(k=3,4,\cdots),
\end{equation*}
where $(\mathcal{M}_{\lambda}^{k-1})^{\perp}=\{u\in\mathcal{M}_{\lambda}:\langle u,v\rangle_\lambda=0\text{ for all }v\in\bigoplus_{i=1}^{k-1}\mathcal{N}_{\lambda,i}\}$ and $\mathcal{N}_{\lambda,i}=\text{span}\{u\in \mathcal{M}_\lambda:\|u\|_\lambda^2=\beta_i(\lambda)\}$, then by iterating, we can obtain the following lemma.
\begin{lemma}\label{lem0005}
Suppose that the conditions $(B_1)$--$(B_3)$ hold and $k\in\bbn$ with $k\geq3$.
\begin{enumerate}
\item[$(1)$] If $\lambda>\Lambda_k:=\frac{(\max\{-a_0, 0\}\beta_k^0)^2B_0^4-b_0}{b_\infty}$, then $\beta_k(\lambda)$ is well defined and can be attained for some $e_k(\lambda)\in E_{\lambda}$.  Moreover, $(e_k(\lambda), \beta_k(\lambda))$ satisfies \eqref{eq0022} and $(e_k(\lambda),\beta_k(\lambda))\to(e_k^0,\beta_k^0)$ strongly in $\h$ as $\lambda\to+\infty$ up to a subsequence for some $e_k^0\in\mathcal{N}_k$.
\item[$(2)$] There exists $\Lambda_k^*\geq\Lambda_k$ such that $\beta_{k-1}(\lambda)<\beta_k(\lambda)$, $\bigoplus_{i=1}^{k-1}\mathcal{N}_{\lambda,i}\perp\mathcal{N}_{\lambda,k}$ and dim$(\mathcal{N}_{\lambda,k})\leq$dim$(\mathcal{N}_k)$ for $\lambda>\Lambda_k^*$.
\end{enumerate}
\end{lemma}

By Lemmas~\ref{lem0001}, \ref{lem0003} and \ref{lem0005}, $\bigoplus_{i=1}^{k_0^*-1}\mathcal{N}_{\lambda,i}$ and $(\mathcal{M}_{\lambda}^{k_0^*-1})^{\perp}$ are well defined for $\lambda>\Lambda^*_{k_0^*}$, where $k_0^*$ is given by \eqref{eq9998}.
\begin{lemma}\label{lem0006}
Suppose that the conditions $(B_1)$--$(B_3)$ hold.  If $\beta_{k_0^*-1}^0<1$,
then we have
\begin{enumerate}
\item[$(1)$] $\mathcal{D}_\lambda(u,u)\leq\bigg(1-\frac{1}{\beta_{k_0^*-1}(\lambda)}\bigg)\|u\|_{\lambda}^2
    =\bigg(1-\frac{1}{\beta_{k_0^*-1}^0}+o_\lambda(1)\bigg)\|u\|_{\lambda}^2$ in $\bigoplus_{i=1}^{k_0^*-1}\mathcal{N}_{\lambda,i}$;
\item[$(2)$] $\mathcal{D}_\lambda(u,u)\geq\bigg(1-\frac{1}{\beta_{k_0^*}(\lambda)}\bigg)\|u\|_{\lambda}^2
    =\bigg(1-\frac{1}{\beta_{k_0^*}^0}+o_\lambda(1)\bigg)\|u\|_{\lambda}^2$ in $(\mathcal{M}_{\lambda}^{k_0^*-1})^{\perp}$.
\end{enumerate}
\end{lemma}
\begin{proof}
The proof follows immediately from Lemmas~\ref{lem0001}, \ref{lem0003} and \ref{lem0005}.
\end{proof}

\begin{remark}\label{rmk0001}{\em
By Lemmas~\ref{lem0002}, \ref{lem0004} and \ref{lem0005}, we also have $\bigoplus_{i=1}^{k_0^*-1}\mathcal{N}_{\lambda,i}=\emptyset$ in the case of $\beta_1^0>1$ while $\bigoplus_{i=1}^{k_0^*-1}\mathcal{N}_{\lambda,i}\not=\emptyset$ and dim$(\bigoplus_{i=1}^{k_0^*-1}\mathcal{N}_{\lambda,i})\leq$dim$(\bigoplus_{i=1}^{k_0^*-1}\mathcal{N}_{i})$ in the case of $\beta_1^0<1$.}
\end{remark}

\section{The existence of nontrivial solutions}
We first consider the case of $\min\{a_0, b_0\}<0$.  Due to the decomposition of $E_\lambda$ in this case, we will obtain the nonzero critical points of $\mathcal{E}_\lambda(u)$ by using the linking method.
\begin{lemma}\label{lem0007}
Suppose that the conditions $(B_1)$--$(B_3)$ and $(F_1)$--$(F_2)$ hold with $p>2$ and $\min\{a_0, b_0\}<0$.  If $\beta_{k_0^*-1}^0<1$ and $l_0d_0<(1-\frac{1}{\beta_{k_0^*}^0})$ then there exists $\overline{\Lambda}_0>0$ such that
\begin{equation*}
\inf_{(\mathcal{M}_{\lambda}^{k_0^*-1})^{\perp}\cap\mathbb{B}_{\lambda,\rho_0}}\mathcal{E}_\lambda(u)\geq\kappa_0\quad\text{and}\quad
\sup_{\partial \mathcal{Q}_{\lambda,R_0}}\mathcal{E}_\lambda(u)\leq0
\end{equation*}
for all $\lambda>\overline{\Lambda}_0$ with some $\kappa_0>0$ and $R_0>\rho_0>0$ independent of $\lambda$, where $\mathbb{B}_{\lambda,\rho_0}=\{u\in E_\lambda:\|u\|_\lambda=\rho_0\}$, $\mathcal{Q}_{\lambda,R_0}=\{u=v+t e_{k_0^*}(\lambda): v\in\bigoplus_{i=1}^{k_0^*-1}\mathcal{N}_{\lambda,i}, t\geq0, \|u\|_\lambda\leq R_0\}$ and
\begin{equation*}
d_0=\left\{\aligned&\mathcal{A}_\infty a_0^{-1},\quad\text{if }a_0>0,\\
&4\mathcal{A}_\infty^2 B^4_0,\quad\text{if }a_0\leq0.\endaligned\right.
\end{equation*}
\end{lemma}
\begin{proof}
Since $l_0 d_0<(1-\frac{1}{\beta_{k_0^*}^0})$, there exists $\delta>0$ such that $(1+\delta)l_0 d_0<(1-\frac{1}{\beta_{k_0^*}^0})$.  By the conditions $(F_1)$--$(F_2)$, we have $|F(u)|\leq\frac{(1+\delta)l_0}{2}|u|^2+C|u|^{2^*}$ for all $u\in\bbr$.  It follows from the Sobolev inequality, \eqref{eq0020} and \eqref{eq0010} that
\begin{eqnarray*}
\mathcal{E}_\lambda(u)&\geq&\frac12(\mathcal{D}_\lambda(u,u)-(1+\delta_0)l_0\|u\|_{L^2(\bbr^N)}^2)-C\|u\|_{L^{2^*}(\bbr^N)}^{2^*}\\
&\geq&\frac12(\mathcal{D}_\lambda(u,u)-(1+\delta)l_0 C_\lambda\|u\|_{\lambda}^2)-CS^{-\frac{2^*}{2}}B_0^{2^*}C_\lambda^{\frac{2^*}{4}}\|u\|_{\lambda}^{2^*}\\
&=&\frac12(\mathcal{D}_\lambda(u,u)-(1+\delta)l_0 (d_0+o_\lambda(1))\|u\|_{\lambda}^2)-CS^{-\frac{2^*}{2}}B_0^{2^*}(d_0+o_\lambda(1))^{\frac{2^*}{4}}\|u\|_{\lambda}^{2^*}
\end{eqnarray*}
for all $u\in E_\lambda$.  Using Lemma~\ref{lem0006}, we get
\begin{equation*}
\mathcal{E}_\lambda(u)\geq\frac{1}{2}\bigg(1-\frac{1}{\beta_{k_0^*}^0}+o_\lambda(1)-(1+\delta)l_0 d_0\bigg)\|u\|_{\lambda}^2-CS^{-\frac{2^*}{2}}B_0^{2^*}(d_0+o_\lambda(1))^{\frac{2^*}{4}}\|u\|_{\lambda}^{2^*}
\end{equation*}
for all $u\in (\mathcal{M}_{\lambda}^{k_0^*-1})^{\perp}$.  Now, by a standard argument, it is easy to check that there exists $\Lambda_1^*>\Lambda_{k_0^*}^*$ such that
\begin{equation*}
\inf_{(\mathcal{M}_{\lambda}^{k_0^*-1})^{\perp}\cap\mathbb{B}_{\lambda,\rho_0}}\mathcal{E}_\lambda(u)\geq\kappa_0
\end{equation*}
for all $\lambda>\Lambda^*_1$ with some $\kappa_0>0$ and $\rho_0>0$ independent of $\lambda$.  It remains to show that there exists a positive constant $R_0(>\rho_0)$ so large that
\begin{equation*}
\sup_{\partial \mathcal{Q}_{\lambda,R_0}}\mathcal{E}_\lambda(u)\leq0.
\end{equation*}
for $\lambda$ sufficient large.  Indeed,
let $u_\lambda\in\partial \mathcal{Q}_{\lambda,R}$ be such that $u_\lambda=R\widetilde{u}$ with $\widetilde{u}_\lambda\in\mathcal{Q}_{\lambda,1}$, then one of the following two cases must happen:
\begin{enumerate}
\item[$(1)$] $\widetilde{u}_\lambda\in \bigoplus_{i=1}^{k_0^*-1}\mathcal{N}_{\lambda,i}$ and $\|\widetilde{u}_\lambda\|_\lambda\leq1$;
\item[$(2)$] $\widetilde{u}_\lambda\in \mathcal{Q}_{\lambda,1}\backslash\bigoplus_{i=1}^{k_0^*-1}\mathcal{N}_{\lambda,i}$ and $\|\widetilde{u}_\lambda\|_\lambda=1$.
\end{enumerate}
In the case~$(1)$, it follows from Lemma~\ref{lem0006} and the condition $(F_4)$ that $\mathcal{E}_\lambda(R\widetilde{u}_\lambda)\leq0$ for all $R\geq0$ and $\lambda>\Lambda_1^*$.  In the case~$(2)$, also by using Lemma~\ref{lem0006}, we deduce
\begin{eqnarray}
\mathcal{E}_\lambda(R\widetilde{u}_\lambda)&=&R^2\bigg(\frac12\mathcal{D}_\lambda(\widetilde{u}_\lambda,\widetilde{u}_\lambda)
-\int_{\bbr^N}\frac{F(x,R\widetilde{u}_\lambda)}{(R\widetilde{u}_\lambda)^2}\widetilde{u}_\lambda^2dx\bigg)\notag\\
&\leq&R^2\bigg(\frac12(1-\frac{1}{\beta_{k_0^*}})+o_\lambda(1)
-\int_{\bbr^N}\frac{F(x,R\widetilde{u}_\lambda)}{(R\widetilde{u}_\lambda)^2}\widetilde{u}_\lambda^2dx\bigg).\label{eq0023}
\end{eqnarray}
On the other hand, since $\widetilde{u}_\lambda\in\bigoplus_{i=1}^{k_0^*}\mathcal{N}_{\lambda,i}$, by Lemmas~\ref{lem0001}, \ref{lem0003} and \ref{lem0005}, we have $\widetilde{u}_\lambda=\widetilde{u}+o_\lambda(1)$ strongly in $\h$ for some $\widetilde{u}\in \bigoplus_{i=1}^{k_0^*}\mathcal{N}_{i}$ with
\begin{equation*}
\int_{\Omega}|\Delta \widetilde{u}|^2+\max\{a_0, 0\}|\nabla \widetilde{u}|^2+\max\{b_0, 0\}|\widetilde{u}|^2dx=1,
\end{equation*}
which together with the conditions $(F_1)$--$(F_2)$ gives that
\begin{equation}\label{eq9997}
\int_{\bbr^N}\frac{F(x,R\widetilde{u}_\lambda)}{(R\widetilde{u}_\lambda)^2}\widetilde{u}_\lambda^2dx
=\int_{\bbr^N}\frac{F(x,R\widetilde{u})}{(R\widetilde{u})^2}\widetilde{u}^2dx+o_\lambda(1).
\end{equation}
Clearly,
\begin{equation*}
\|u\|_{\Omega,0}=\bigg(\int_{\Omega}\big(|\Delta u|^2+\max\{a_0, 0\}|\nabla u|^2+\max\{b_0, 0\}|u|^2\big)dx\bigg)^{\frac12}
\end{equation*}
is a norm in $H$, and note that dim$(\bigoplus_{i=1}^{k_0^*}\mathcal{N}_{i})<+\infty$, we see that there exists $d_*>0$ such that
\begin{equation}\label{eq9996}
\|u\|_{\Omega,0}\leq d_*\|u\|_{L^2(\bbr^N)}\quad\text{for all }u\in\bigoplus_{i=1}^{k_0^*}\mathcal{N}_{i}.
\end{equation}
Now, by the condition $(F_2)$ and the Fatou lemma, there exists $R_0>\rho_0$ independent of $\lambda$ such that
\begin{equation*}
\int_{\bbr^N}\frac{F(x,R_0\widetilde{u})}{R_0^2\widetilde{u}^2}\widetilde{u}^2dx\geq(1-\frac{1}{\beta_{k_0^*}}).
\end{equation*}
It follows from \eqref{eq0023} and \eqref{eq9997}  that there exist $\overline{\Lambda}_0>\Lambda^*_1$ and such that $\sup_{\partial \mathcal{Q}_{\lambda,R_0}}\mathcal{E}_\lambda(u)\leq0$ for all $\lambda>\overline{\Lambda}_0$.
\end{proof}

\begin{lemma}\label{lem0008}
Suppose that the conditions $(B_1)$--$(B_3)$ and $(F_1)$--$(F_2)$ hold with $p=2$.  If
\begin{equation*}
l_0 d_0<(1-\frac{1}{\beta_{k_0^*}^0})<\frac{l_\infty}{d_*},
\end{equation*}
then there exists $\widetilde{\Lambda}_0>0$ such that the conclusions of Lemma~\ref{lem0007} hold for $\lambda>\widetilde{\Lambda}_0$, where $d_*$ is given by \eqref{eq9996}.
\end{lemma}
\begin{proof}
Similarly as in the proof of Lemma~\ref{lem0007}, we can show that there exists $\Lambda^*_2>\Lambda_{k_0^*}^*$ such that
\begin{equation*}
\inf_{(\mathcal{M}_{\lambda}^{k_0^*-1})^{\perp}\cap\mathbb{B}_{\lambda,\widetilde{\rho}_0}}\mathcal{E}_\lambda(u)\geq\widetilde{\kappa}_0
\end{equation*}
for all $\lambda>\Lambda^*_2$ with some $\widetilde{\kappa}_0>0$ and $\widetilde{\rho}_0>0$ independent of $\lambda$.  In what follows, we will prove that there exists a positive constant $\widetilde{R}_0(>\widetilde{\rho}_0)$ so large that
\begin{equation*}
\sup_{\partial \mathcal{Q}_{\lambda,\widetilde{R}_0}}\mathcal{E}_\lambda(u)\leq0.
\end{equation*}
for $\lambda$ sufficient large.  In fact,  if $u_\lambda\in\partial \mathcal{Q}_{\lambda,R}$ is such that $u_\lambda=R\widetilde{u}$ with $\widetilde{u}_\lambda\in\mathcal{Q}_{\lambda,1}$ then one of the following two cases must happen:
\begin{enumerate}
\item[$(1)$] $\widetilde{u}_\lambda\in \bigoplus_{i=1}^{k_0^*-1}\mathcal{N}_{\lambda,i}$ and $\|\widetilde{u}_\lambda\|_\lambda\leq1$;
\item[$(2)$] $\widetilde{u}_\lambda\in \mathcal{Q}_{\lambda,1}\backslash\bigoplus_{i=1}^{k_0^*-1}\mathcal{N}_{\lambda,i}$ and $\|\widetilde{u}_\lambda\|_\lambda=1$.
\end{enumerate}
In the case~$(1)$, it follows from Lemma~\ref{lem0006} and the condition $(F_3)$ that $\mathcal{E}_\lambda(R\widetilde{u}_\lambda)\leq 0$ for all $R\geq0$ and $\lambda>\Lambda^*_2$.  In the case~$(2)$, by similar arguments as used in the proofs of \eqref{eq0023} and Lemma~\ref{lem0005}, we have
\begin{eqnarray}\label{eq9994}
\mathcal{E}_\lambda(R\widetilde{u})
\leq R^2\bigg(\frac12(1-\frac{1}{\beta_{k_0^*}})+o_\lambda(1)
-\int_{\bbr^N}\frac{F(x,R\widetilde{u}_\lambda)}{(R\widetilde{u}_\lambda)^2}\widetilde{u}_\lambda^2dx\bigg)
\end{eqnarray}
and
\begin{equation}\label{eq9995}
\int_{\bbr^N}\frac{F(x,R\widetilde{u}_\lambda)}{(R\widetilde{u}_\lambda)^2}\widetilde{u}_\lambda^2dx
=\int_{\bbr^N}\frac{F(x,R\widetilde{u})}{(R\widetilde{u})^2}\widetilde{u}^2dx+o_\lambda(1)
\end{equation}
for some $\widetilde{u}\in \bigoplus_{i=1}^{k_0^*}\mathcal{N}_{i}$ with
\begin{equation*}
\int_{\Omega}\big(|\Delta \widetilde{u}|^2+\max\{a_0, 0\}|\nabla \widetilde{u}|^2+\max\{b_0, 0\}|\widetilde{u}|^2\big)dx=1.
\end{equation*}
Thanks to the Fatou lemma, we deduce from the condition $(F_2)$ and \eqref{eq9996} that
\begin{equation*}
\lim_{R\to+\infty}\int_{\bbr^N}\frac{F(x,R\widetilde{u})}{(R\widetilde{u})^2}\widetilde{u}^2dx
\geq\int_{\bbr^N}\lim_{R\to+\infty}\frac{F(x,R\widetilde{u})}{R^2\widetilde{u}^2}\widetilde{u}^2dx=\frac{l_\infty}{2}\|\widetilde{u}\|_{L^2(\bbr^N)}^2
\geq\frac{l_\infty}{2 d_*}.
\end{equation*}
Since $\frac{l_\infty}{ d_*}>1-\frac{1}{\beta_{k_0^*}}$, by \eqref{eq9994} and \eqref{eq9995}, there exist $\widetilde{\Lambda}_0\geq\Lambda^*_2$ and $\widetilde{R}_0>\widetilde{\rho}_0$ independent of $\lambda$ such that $\sup_{\partial \mathcal{Q}_{\lambda,R_0}}\mathcal{E}_\lambda(u)\leq0$ for all $\lambda>\widetilde{\Lambda}_0$.
\end{proof}

By Lemmas~\ref{lem0007} and \ref{lem0008}, we know that $\mathcal{E}_\lambda(u)$ has a linking structure in $E_\lambda$ for all$\lambda>\max\{\overline{\Lambda}_0, \widetilde{\Lambda}_0\}$ in the case of $\min\{a_0, b_0\}<0$.  By the well known linking theorem, $\mathcal{E}_\lambda(u)$ has a Cerami sequence at level $c_\lambda$ ($(C)_{c_\lambda}$ sequence for short) for all $\lambda>\max\{\overline{\Lambda}_0, \widetilde{\Lambda}_0\}$, that is, there exists $\{u_{\lambda,n}\}\subset E_\lambda$ with $\lambda>\max\{\overline{\Lambda}_0, \widetilde{\Lambda}_0\}$ such that
\begin{equation*}
\mathcal{E}_\lambda(u_{\lambda,n})=c_\lambda+o_n(1)\quad\text{and}\quad (1+\|u_{\lambda,n}\|_\lambda)\mathcal{E}_\lambda'(u_{\lambda,n})=o_n(1)\text{ strongly in }E_\lambda^*.
\end{equation*}
In the special case $k_0^*=1$, the linking structure is actually the mountain pass geometry and the linking theorem can be replaced by the well known mountain pass theorem.
Moreover, due to the conditions $(F_3)$ and $(F_4)$, we can see that $c_\lambda\in[\min\{\kappa_0, \widetilde{\kappa}_0\}, \frac12\max\{R_0^2, \widetilde{R}_0^2\}]$.

We next consider the case of $\min\{a_0,b_0\}\geq0$.
\begin{lemma}\label{lem0020}
Suppose that the conditions $(F_1)$--$(F_2)$ hold and $\min\{a_0,b_0\}\geq0$.  Then there exists $l_\infty^*>0$ such that
\begin{equation*}
\inf_{\mathbb{B}_{\rho_0^*}}\mathcal{E}_\lambda(u)\geq\kappa_0^*\quad\text{and}\quad\mathcal{E}_\lambda(R_0^*\phi_1)\leq0
\end{equation*}
hold for $\lambda>\max\{0, -\frac{b_0}{b_\infty}\}$ with some $R_0^*>\rho_0^*>0$ and $\kappa_0^*>0$ independent of $\lambda$ in the following two cases:
\begin{enumerate}
\item[$(a)$] $p>2$;
\item[$(b)$] $p=2$ and $l_\infty>l_\infty^*$.
\end{enumerate}
\end{lemma}
\begin{proof}
Since $\mathcal{D}_\lambda(u,v)=\langle u,v\rangle_{\lambda}$ for all $(u,v)\in E_\lambda$ and $\mathcal{D}_\lambda(u,u)$ is definite on $E_\lambda$ for $\lambda>\max\{0, -\frac{b_0}{b_\infty}\}$ in the case of $\min\{a_0,b_0\}\geq0$, we can get the conclusions by similar but more simple arguments used in the proofs of  Lemmas~\ref{lem0007} and \ref{lem0008}.
\end{proof}

Due to Lemma~\ref{lem0020}, we can see that $\mathcal{E}_\lambda(u)$ has a mountain pass geometry for $\lambda>\max\{0, -\frac{b_0}{b_\infty}\}$ in the case of $\min\{a_0,b_0\}\geq0$.  It follows from the well known mountain pass theorem, $\mathcal{E}_\lambda(u)$ has a $(C)_{c_\lambda}$ sequence for all $\lambda>\max\{0, -\frac{b_0}{b_\infty}\}$.  Furthermore, $c_\lambda\in[\kappa_0^*, \frac12(R_0^*)^2(\overline{\mu}_1^2+a_0\overline{\mu}_1+b_0)]$.  In another word, there exists $\overline{\Lambda}_{*,0}>0$ such that $\mathcal{E}_\lambda(u)$ always has a $(C)_{c_\lambda}$ sequence for $\lambda>\overline{\Lambda}_{*,0}$ and $c_\lambda\in[C, C']$ in both cases of $\min\{a_0,b_0\}\geq0$ and $\min\{a_0,b_0\}<0$.
\begin{lemma}\label{lem0009}
Suppose that the conditions $(B_1)$--$(B_3)$ and $(F_1)$--$(F_2)$ and $(F_4)$ hold with $p>2$.  Then there exist $\overline{\Lambda}_1>\overline{\Lambda}_{*,0}$ and $C_0>0$ independent of $\lambda$ such that $\|u_{\lambda,n}\|_\lambda\leq C_0+o_n(1)$ for all $\lambda>\overline{\Lambda}_1$.
\end{lemma}
\begin{proof}
By the condition $(F_4)$, we have
\begin{eqnarray*}
o_n(1)+c_\lambda&\geq&\mathcal{E}_{\lambda}(u_{\lambda,n})-\frac{1}{2}\langle\mathcal{E}_{\lambda}(u_{\lambda,n}), u_{\lambda,n}\rangle_{E_\lambda^*, E_\lambda}\\
&=&\frac12\int_{\bbr^N}(f(u_{\lambda,n})u_{\lambda,n}-2F(u_{\lambda,n}))dx\\
&\geq&\frac{l_*}{2}\|u_{\lambda,n}\|_{L^p(\bbr^N)}^p.
\end{eqnarray*}
On the other hand, due to the conditions $(B_1)$--$(B_2)$, for all $\lambda>\overline{\Lambda}_{*,0}$, we get from the H\"older and the Gagliardo-Nirenberg inequalities that
\begin{eqnarray*}
\mathcal{G}_\lambda(u_{\lambda,n},u_{\lambda,n})&\leq&\max\{-a_0, 0\}B_0^2\|\Delta u_{\lambda,n}\|_{L^2(\bbr^N)}\|u_{\lambda,n}\|_{L^2(\bbr^N)}\\
&&+\max\{-b_0, 0\}|\mathcal{B}_\infty|^{\frac{p-2}{p}}\|u_{\lambda,n}\|_{L^p(\bbr^N)}^2\\
&\leq&\max\{-a_0, 0\}B_0^2|\mathcal{B}_\infty|^{\frac{p-2}{2p}}\|\Delta u_{\lambda,n}\|_{L^2(\bbr^N)}\|u_{\lambda,n}\|_{L^p(\bbr^N)}\\
&&+\max\{-a_0, 0\}B_0^2(\frac{1}{\lambda b_\infty+b_0})^{\frac12}\|u_{\lambda,n}\|_\lambda^2+\max\{-b_0, 0\}|\mathcal{B}_\infty|^{\frac{p-2}{p}}\|u_{\lambda,n}\|_{L^p(\bbr^N)}^2\\
&\leq&(\frac12+\max\{-a_0, 0\}B_0^2(\frac{1}{\lambda b_\infty+b_0})^{\frac12})\|u_{\lambda,n}\|_\lambda^2\\
&&+(2\max\{-a_0, 0\}^2B_0^4+\max\{-b_0, 0\})|\mathcal{B}_\infty|^{\frac{p-2}{p}}\|u_{\lambda,n}\|_{L^p(\bbr^N)}^2.
\end{eqnarray*}
It implies from the conditions $(B_2)$, $(F_1)$--$(F_2)$ and the H\"older inequality that for all $\lambda>\overline{\Lambda}_{*,0}$,
\begin{eqnarray*}
|\int_{\bbr^N}F(u_{\lambda,n})dx|&\leq&2l_0\|u_{\lambda,n}\|_{L^2(\bbr^N)}^2+C\|u_{\lambda,n}\|_{L^p(\bbr^N)}^p\\
&\leq&\frac{2l_0}{\lambda b_\infty+b_0}\|u_{\lambda,n}\|_\lambda^2+2l_0|\mathcal{B}_\infty|^{\frac{p-2}{p}}\|u_{\lambda,n}\|_{L^p(\bbr^N)}^2+C\|u_{\lambda,n}\|_{L^p(\bbr^N)}^p.
\end{eqnarray*}
Now, we can see from $\mathcal{E}_\lambda(u_{\lambda,n})=c_\lambda+o_n(1)$ that
\begin{equation*}
(\frac12+o_\lambda(1))\|u_{\lambda,n}\|_\lambda^2\leq C'c_\lambda+o_n(1).
\end{equation*}
Note that $c_\lambda\in[C, C']$, there exist $\overline{\Lambda}_1>\overline{\Lambda}_{*,0}$ and $C_0>0$ independent of $\lambda$ such that $\|u_{\lambda,n}\|_\lambda\leq C_0+o_n(1)$ for all $\lambda>\overline{\Lambda}_1$.
\end{proof}

\begin{lemma}\label{lem0010}
Suppose that the conditions $(B_1)$--$(B_3)$ and $(F_1)$--$(F_3)$ hold with $p=2$.  If $l_\infty\not\in\sigma(\Delta^2-a_0\Delta+b_0, L^2(\Omega))$, then there exists $\widetilde{\Lambda}_1>\overline{\Lambda}_{*,0}$ such that $\{u_{\lambda,n}\}$ is bounded in $E_\lambda$ for all $\lambda>\widetilde{\Lambda}_1$.
\end{lemma}
\begin{proof}
Suppose on the contrary that there exists a subsequence of $\{u_{\lambda,n}\}$, which is still denoted by $\{u_{\lambda,n}\}$, such that $\|u_{\lambda,n}\|_\lambda\to+\infty$ as $n\to+\infty$.  Let $w_{\lambda,n}=\frac{u_{\lambda,n}}{\|u_{\lambda,n}\|_\lambda}$.  Then without loss of generality, we may assume that $w_{\lambda,n}\rightharpoonup w_{\lambda, 0}$ weakly in $E_\lambda$ for some $w_{\lambda, 0}\in E_\lambda$ as $n\to\infty$.

{\bf Claim~1:}\quad There exists $\Lambda^*_3>\overline{\Lambda}_{*,0}$ such that $w_{\lambda, 0}\not=0$.

Indeed, if $w_{\lambda, 0}=0$, then by Remark~\ref{rmk0001}, we can see that $w_{\lambda,n}^-\to0$ strongly in $E_\lambda$ with $\lambda>\overline{\Lambda}_{*,0}$ as $n\to\infty$, where $w_{\lambda,n}^-$ is the projection of $w_{\lambda,n}$ in $\bigoplus_{i=1}^{k_0^*-1}\mathcal{N}_{\lambda,i}$.  It follows from Lemma~\ref{lem0006} that
\begin{equation}\label{eq0024}
\mathcal{D}_\lambda(w_{\lambda,n},w_{\lambda,n})\geq\bigg(1-\frac{1}{\beta_{k_0^*}^0}+o_\lambda(1)\bigg)\|w_{\lambda,n}^+\|_\lambda^2+o_n(1),
\end{equation}
where $w_{\lambda,n}^+=w_{\lambda,n}-w_{\lambda,n}^-$.  On the other hand, thanks to the conditions $(F_2)$--$(F_3)$, we see from $(1+\|u_{\lambda,n}\|_\lambda)\mathcal{E}_\lambda'(u_{\lambda,n})=o_n(1)$ strongly in $E_\lambda^*$ that
\begin{equation*}
\mathcal{D}_\lambda(w_{\lambda,n},w_{\lambda,n})=o_n(1)+\int_{\bbr^N}\frac{f(u_{\lambda,n})}{u_{\lambda,n}}w_{\lambda,n}^2dx\leq o_n(1)+l_\infty\|w_{\lambda,n}\|_{L^2(\bbr^N)}^2,
\end{equation*}
which, together with the Sobolev embedding theorem, the fact that $E_\lambda$ is continuously embedded into $\h$ for $\lambda>\max\{0, \frac{-b_0}{b_\infty}\}$ and the condition $(B_2)$, implies that
\begin{eqnarray}\label{eq0025}
\mathcal{D}_\lambda(w_{\lambda,n},w_{\lambda,n})\leq\frac{4l_\infty}{\lambda b_\infty+b_0}\|w_{\lambda,n}^+\|_\lambda^2+o_n(1).
\end{eqnarray}
Thanks to Lemma~\ref{lem0005}, we can deduce from \eqref{eq0024} and \eqref{eq0025} that there exists $\Lambda^*_3>\overline{\Lambda}_{*,0}$ such that $w_{\lambda,n}^+\to0$ strongly in $E_\lambda$ as $n\to\infty$ for $\lambda>\Lambda^*_3$, which is inconsistent with $\|w_{\lambda,n}\|_\lambda=1$ for all $n\in\bbn$.

{\bf Claim~2:}\quad There exists $\Lambda^*_4>\Lambda^*_3$ such that $w_{\lambda,n}\to w_{\lambda, 0}$ strongly in $E_\lambda$ as $n\to\infty$ for $\lambda>\Lambda^*_4$ up to a subsequence.

In fact, let $\mathcal{Q}_{\lambda,0}=\{x\in\bbr^N: w_{\lambda, 0}\not=0\}$, then $|u_{\lambda, n}|\to+\infty$ as $n\to\infty$ on $\mathcal{Q}_{\lambda,0}$.  It follows from the conditions $(F_2)$--$(F_3)$ and a variant of the
Lebesgue dominated convergence theorem (cf. \cite[Theorem~2.2]{PK74}) that
\begin{eqnarray*}
\lim_{n\to\infty}\int_{\bbr^N}\frac{f(u_{\lambda, n})v}{\|u_{\lambda, n}\|_\lambda}dx=\lim_{n\to\infty}\int_{\bbr^N}\frac{f(u_{\lambda, n})}{u_{\lambda, n}}w_{\lambda, n}v\chi_{\mathcal{Q}_{\lambda,0}}dx
=\int_{\bbr^N}l_\infty w_{\lambda, 0}vdx
\end{eqnarray*}
for every $v\in\h$.  Since $w_{\lambda,n}\rightharpoonup w_{\lambda, 0}$ weakly in $E_\lambda$ as $n\to\infty$, due to the fact that $E_\lambda$ is continuously embedded into $\h$ for $\lambda>\max\{0, \frac{-b_0}{b_\infty}\}$, we have that $\lim_{n\to\infty}\mathcal{D}_\lambda(w_{\lambda,n},v)=\mathcal{D}_\lambda(w_{\lambda, 0},v)$ for all $v\in\h$.
Thus, $w_{\lambda, 0}\in\h$ satisfies the following equation in the weak sense:
\begin{eqnarray}\label{eq0027}
\Delta^2 w_{\lambda, 0}-a_0\Delta w_{\lambda, 0}+(\lambda b(x)+b_0)w_{\lambda, 0}=l_\infty w_{\lambda, 0},\quad\text{in }\bbr^N.
\end{eqnarray}
Let $I_\lambda(u)=\frac12(\mathcal{D}_{\lambda}(u,u)-l_\infty\|u\|_{L^2(\bbr^N)^2})$.  Then $I_\lambda'(w_{\lambda, 0})=0$ in $E_\lambda^*$.
Now, by Remark~\ref{rmk0001} and a similar argument used in the proof of \eqref{eq0014}, we have
\begin{eqnarray}
o_n(1)&=&\langle\frac{\mathcal{E}_\lambda'(u_{\lambda, n})}{\|u_{\lambda,n}\|_\lambda}-I_\lambda'(w_{\lambda, 0}), w_{\lambda, n}-w_{\lambda, 0}\rangle_{E_\lambda^*,E_\lambda}\notag\\
&\geq&\bigg(1-\frac{1}{\beta_{k_0^*}^0}+o_\lambda(1)\bigg)\|w_{\lambda, n}-w_{\lambda, 0}\|_\lambda^2+o_n(1)-\int_{\bbr^N}\frac{f(u_{\lambda, n})}{u_{\lambda, n}}|w_{\lambda, n}-w_{\lambda, 0}|^2dx\notag\\
&&-\int_{\bbr^N}(\frac{f(u_{\lambda, n})}{u_{\lambda, n}}-l_\infty)(w_{\lambda, n}-w_{\lambda, 0})w_{\lambda, 0}dx.\label{eq9993}
\end{eqnarray}
Note that $(w_{\lambda, n}-w_{\lambda, 0})w_{\lambda, 0}=o_n(1)$ strongly in $L^1(\bbr^N)$.  By the conditions $(F_2)$--$(F_3)$, the Sobolev embedding theorem and a variant of the
Lebesgue dominated convergence theorem (cf. \cite[Theorem~2.2]{PK74}), we can see that
\begin{eqnarray}
&&\int_{\bbr^N}\frac{f(u_{\lambda, n})}{u_{\lambda, n}}|w_{\lambda, n}-w_{\lambda, 0}|^2dx+\int_{\bbr^N}(\frac{f(u_{\lambda, n})}{u_{\lambda, n}}-l_\infty)(w_{\lambda, n}-w_{\lambda, 0})w_{\lambda, 0}dx\notag\\
&&\leq\frac{l_\infty}{\lambda b_\infty+b_0}\|w_{\lambda, n}-w_{\lambda, 0}\|_\lambda^2+o_n(1).\label{eq9992}
\end{eqnarray}
Combining \eqref{eq9993} and \eqref{eq9992}, we must have that there exists $\Lambda^*_4>\Lambda^*_3$ such that $w_{\lambda,n}\to w_{\lambda, 0}$ strongly in $E_\lambda$ as $n\to\infty$ for $\lambda>\Lambda^*_4$.

{\bf Claim~3:} \quad $w_{\lambda,0}\to w_{\infty,0}$ strongly in $H^1(\bbr^N)$ as $\lambda\to+\infty$ up to a subsequence for some $w_{\infty, 0}\in H$ which satisfies the following equation in the weak sense:
\begin{equation}\label{eq0026}
\Delta^2 w_{\infty, 0}-a_0\Delta w_{\infty, 0}+b_0w_{\infty, 0}=l_\infty w_{\infty, 0},\quad\text{in }\Omega.
\end{equation}

Indeed, since $\|w_{\lambda,n}\|_\lambda=1$, by \eqref{eq0020} and \eqref{eq0010} and a similar argument used in the proof of Lemma~\ref{lem0001}, we can show that $w_{\lambda, 0}\rightharpoonup w_{\infty, 0}$ weakly in $\h$ and $w_{\lambda, 0}\to w_{\infty, 0}$ strongly in $H^1(\bbr^N)$ for some $w_{\infty, 0}\in H$ with $w_{\infty, 0}=0$ outside $\Omega$ as $\lambda\to+\infty$, up to a subsequence.  It follows from \eqref{eq0027} that $w_{\infty, 0}\in H$ satisfies \eqref{eq0026} in the weak sense.

Now, multiplying respectively \eqref{eq0026} and \eqref{eq0027} with $w_{\infty, 0}$ and $w_{\lambda, 0}$, and integrating, we can see that
\begin{eqnarray*}
&&\|\Delta w_{\infty, 0}\|_{L^2(\bbr^N)}^2+a_0\|\nabla w_{\infty, 0}\|_{L^2(\bbr^N)}^2+b_0\|w_{\infty, 0}\|_{L^2(\bbr^N)}^2\\
&\leq&\lim_{\lambda\to+\infty}(\|w_{\lambda, 0}\|_\lambda^2+\mathcal{G}_\lambda(w_{\lambda, 0},w_{\lambda, 0}))\\
&=&l_\infty\lim_{\lambda\to+\infty}\|w_{\lambda, 0}\|_{L^2(\bbr^N)}^2\\
&=&l_\infty\|w_{\infty, 0}\|_{L^2(\bbr^N)}^2\\
&=&\|\Delta w_{\infty, 0}\|_{L^2(\bbr^N)}^2+a_0\|\nabla w_{\infty, 0}\|_{L^2(\bbr^N)}^2+b_0\|w_{\infty, 0}\|_{L^2(\bbr^N)}^2.
\end{eqnarray*}
Hence, $\int_{\bbr^N}\lambda b(x)w_{\lambda, 0}^2dx=o_\lambda(1)$ and $w_{\lambda, 0}\to w_{\infty, 0}$ strongly in $H^2(\bbr^N)$ as $\lambda\to+\infty$ up to a subsequence.  Therefore, by Claim~1 and Claim~2, $\|w_{\infty, 0}\|_{\Omega,0}=1$, which gives $w_{\infty, 0}\not=0$.  It follows from Claim~2 that $l_\infty\in\sigma(\Delta^2-a_0\Delta+b_0, L^2(\Omega))$, which contradicts the assumption that $l_\infty\not\in\sigma(\Delta^2-a_0\Delta+b_0, L^2(\Omega))$.  Thus, there exist $\widetilde{\Lambda}_1>\overline{\Lambda}_{*,0}$ such that $\{u_{\lambda,n}\}$ is bounded in $E_\lambda$ for all $\lambda>\widetilde{\Lambda}_1$.
\end{proof}

\begin{remark}{\em
Since span$\{\phi_k\}=H$ and $\phi_k$ are orthogonal in $H$, it is easy to show that $\sigma(\Delta^2-a_0\Delta+b_0, L^2(\Omega))=\{\mu_k^2+a_0\mu_k+b_0\}$.}
\end{remark}

Now, we can give the proofs of Theorems~\ref{thm0001}-\ref{thm0003}.\vspace{6pt}

\noindent\textbf{Proof of Theorem~\ref{thm0001}:}\quad By Lemma~\ref{lem0009}, $u_{\lambda,n}\rightharpoonup u_{\lambda, 0}$ weakly in $E_\lambda$ with $\lambda>\overline{\Lambda}_1$ as $n\to\infty$ up to a subsequence.  Without loss of generality, we may assume that $u_{\lambda,n}\rightharpoonup u_{\lambda, 0}$ weakly in $E_\lambda$ with $\lambda>\overline{\Lambda}_1$ as $n\to\infty$.  Since $\mathcal{E}_\lambda(u)$ is $C^1$, it is easy to see from the fact that $\{u_{\lambda, n}\}$ is a $(C)_{c_\lambda}$ sequence that $\mathcal{E}_\lambda'(u_{\lambda, 0})=0$ in $E_\lambda^*$ with $\lambda>\overline{\Lambda}_1$.  It remains to show that $u_{\lambda, 0}\not=0$ in $E_\lambda$ for $\lambda$ sufficiently large.  Indeed, if $u_{\lambda, 0}=0$, then by the conditions $(B_1)$--$(B_2)$ and $(F_1)$--$(F_2)$ and the the Sobolev, the H\"older, the Gagliardo-Nirenberg inequalities and the fact that $E_\lambda$ is embedded continuously into $\h$ for $\lambda>\max\{0, \frac{-b_0}{b_\infty}\}$ that
\begin{eqnarray*}
\mathcal{G}_\lambda(u_{\lambda,n},u_{\lambda,n})&\leq&\max\{-a_0, 0\}B_0^2\|\Delta u_{\lambda,n}\|_{L^2(\bbr^N)}\|u_{\lambda,n}\|_{L^2(\bbr^N)}+o_n(1)\notag\\
&\leq&\max\{-a_0, 0\}B_0^2\bigg((\frac{1}{\lambda b_\infty+b_0})^{\frac12}+o_n(1)\bigg)\|u_{\lambda,n}\|_\lambda^2+o_n(1)
\end{eqnarray*}
and
\begin{eqnarray*}
|\int_{\bbr^N}F(u_{\lambda,n})dx|&\leq&2l_0\|u_{\lambda,n}\|_{L^2(\bbr^N)}^2+C\|u_{\lambda,n}\|_{L^p(\bbr^N)}^p\notag\\
&\leq&2l_0\bigg((\frac{1}{\lambda b_\infty+b_0})^{\frac12}\|u_{\lambda,n}\|_\lambda^2+o_n(1)\bigg)+C\|\Delta u_{\lambda,n}\|_{L^2(\bbr^N)}^{\frac p2}\|u_{\lambda,n}\|_{L^2(\bbr^N)}^{\frac p2}\notag\\
&\leq&2l_0\bigg((\frac{1}{\lambda b_\infty+b_0})^{\frac12}\|u_{\lambda,n}\|_\lambda^2+o_n(1)\bigg)+C\bigg((\frac{1}{\lambda b_\infty+b_0})^{\frac p2}+o_n(1)\bigg)\|u_{\lambda,n}\|_\lambda^p,
\end{eqnarray*}
which, together with $\langle\mathcal{E}_\lambda(u_{\lambda, n}), u_{\lambda, n}\rangle_{E_{\lambda}^*, E_\lambda}=o_n(1)$ and Lemma~\ref{lem0009}, yields that there exists $\widehat{\Lambda}>\overline{\Lambda}_1$ such that $u_{\lambda, n}\to 0$ strongly in $E_\lambda$ with $\lambda>\widehat{\Lambda}$ as $n\to\infty$.  It follows that $c_\lambda=0$ for $\lambda>\widehat{\Lambda}$.  It is impossible since $c_\lambda\geq C>0$ for all $\lambda>\overline{\Lambda}_1$.
\qquad\raisebox{-0.5mm}{%
\rule{1.5mm}{4mm}}\vspace{6pt}

\noindent\textbf{Proof of Theorem~\ref{thm0002}:}\quad If we can show that $\{u_{\lambda, n}\}$ is uniformly bounded in $E_\lambda$ as Lemma~\ref{lem0009} in this case, that is, there exist $\widetilde{\Lambda}_2>\widetilde{\Lambda}_1$ and $C_0>0$ independent of $\lambda>\widetilde{\Lambda}_2$ such that $\|u_{\lambda, n}\|_\lambda\leq C_0+o_n(1)$ with $\lambda>\widetilde{\Lambda}_2$, then we can follow the proof of Theorem~\ref{thm0001} to obtain the conclusion.  In fact, by Lemma~\ref{lem0010}, there exists $C_\lambda>0$ such that $\|u_{\lambda, n}\|_\lambda\leq C_\lambda$ for all $n\in\bbn$ with $\lambda>\widetilde{\Lambda}_1$.  If $C_\lambda\to+\infty$ as $\lambda\to+\infty$ up to a subsequence, then there exists $\lambda_m\to+\infty$ as $m\to\infty$ and $n_m\in\bbn$ such that $\|u_{\lambda_m, n_m}\|_{\lambda_m}\to+\infty$ as $m\to\infty$.  Let $w_{m}=\frac{u_{\lambda_m, n_m}}{\|u_{\lambda_m, n_m}\|_{\lambda_m}}$, then without loss of generality, we may assume that $w_m\rightharpoonup w_0$ $\h$ as $m\to\infty$ for some $w_0\in\h$ due to \eqref{eq0020} and \eqref{eq0010}.  Now, by using similar arguments in the proof of Lemma~\ref{lem0010}, we can show that $w_0$ is a nontrivial weak solution of \eqref{eq0026}, which is inconsistent with the assumption that $l_\infty\not\in\sigma(\Delta^2-a_0\Delta+b_0, L^2(\Omega))$.
\qquad\raisebox{-0.5mm}{%
\rule{1.5mm}{4mm}}\vspace{6pt}

\noindent\textbf{Proof of Theorem~\ref{thm0003}:}\quad Suppose that $u_\lambda$ is the nontrivial solution of $(\mathcal{P}_\lambda)$ obtained by Theorem~\ref{thm0001} or Theorem~\ref{thm0002} with $\lambda$ large enough.  We can see from Lemma~\ref{lem0009} and the proof of Theorem~\ref{thm0002} that
$\|u_\lambda\|_\lambda\leq C_0$ for all $\lambda$ with some $C_0>0$ independent of $\lambda$.  Now,  similarly as in the proof of Lemma~\ref{lem0010}, we can show that $u_\lambda\to u_*$ strongly in $\h$ as $\lambda\to+\infty$ for some $u_*\in H$ with $u_*\equiv0$ outside $\Omega$,  up to a subsequence.  Furthermore, we also have $\lambda\int_{\bbr^N}b(x)u_\lambda^2dx=o_\lambda(1)$ and $\mathcal{F}'(u_*)=0$ in $H^*$, where $H^*$ is the dual space of $H$ and
\begin{equation*}
\mathcal{F}(u)=\frac12\int_{\Omega}\big(|\Delta u|^2+a_0|\nabla u|^2+b_0u^2\big)dx-\int_{\Omega}F(u)dx.
\end{equation*}
Note that $\mathcal{E}_\lambda(u_\lambda)=c_\lambda\geq C>0$ for all $\lambda>\max\{\overline{\Lambda}_1,\widetilde{\Lambda}_1\}$, we must have $u_*\not=0$ in $H$.  Thus, $u_*$ is a nontrivial weak solution of \eqref{eq0030}.
\qquad\raisebox{-0.5mm}{%
\rule{1.5mm}{4mm}}\vspace{6pt}
\section{Acknowledgements}
Y. Wu thanks Prof. T.-F. Wu for his friendship, encouragement and enlightening discussions.  Y. Wu is supported by the Fundamental Research Funds for the Central Universities (2014QNA67).


\end{document}